\documentclass[12pt]{amsart}
\usepackage{somedefs,tracefnt,latexsym,,rawfonts,epsfig,amsfonts,amssymb,latexsy
m,enumerate,amscd}
\input xypic
\def\cal{\mathcal}
\def\Bbb{\mathbb}

\def\G{\Gamma}
\def\r{\rangle}
\def\l{\langle}
\def\t{\times}
\def\p{\partial}

\newtheorem{thm}{Theorem}[section]

\begin{document}
\title[Vanishing structure set] {Vanishing structure set of $3$-manifolds}
\author[S.K. Roushon]{S.K. Roushon*}
\date{July 06, 2010}
\address{School of Mathematics\\
Tata Institute\\
Homi Bhabha Road\\
Mumbai 400005, India}
\email{roushon@math.tifr.res.in}
\thanks{}
\begin{abstract} In this short note we update a result proved  
in \cite{R4}. This will complete our program of \cite{R0} 
showing that the structure set vanishes for compact  
aspherical $3$-manifolds.\end{abstract}



\maketitle

\section{Introduction} This paper is to note that the program we 
started in \cite{R0} is now complete.

Let us first  
recall that a compact manifold $M$ with boundary is called {\it topologically 
rigid} if any homotopy equivalence $f:(N, \p N)\to (M, \p M)$ from another 
compact manifold with 
boundary, so that 
$f|_{\p N}:\p N\to \p M$ is a homeomorphism is homotopic to a homeomorphism 
relative to boundary.

 Let $M$ be a compact connected  
$3$-manifold whose fundamental group is torsion free.
 
We prove the following theorem.

\begin{thm} \label{main1} If $M$ is aspherical then 
$M\t {\Bbb D}^n$ is topologically 
rigid for $n\geq 2$. Here ${\Bbb D}^n$ 
denotes the $n$-dimensional disc.\end{thm}

In \cite{R0} and \cite{R1} we proved Theorem \ref{main1} under various 
conditions. In \cite{R0} we proved it for the nonempty boundary case and 
for the situation when the manifold contains an incompressible square root 
closed torus. In \cite{R1} we assumed the manifold has positive  
first Betti number. Due to some recent developments in Geometric Topology 
(see \cite{BL}, \cite{BFL}, \cite{R2}, \cite{R3} and \cite{R4})
we are now able to deduce Theorem \ref{main1}. Also the main ideas 
from \cite{R0} and \cite{R1} go behind the proof of this general case.

The first step to prove Theorem \ref{main1} is to show that the Whitehead 
group of $\pi_1(M)$ is trivial. We deduce the following for this purpose.

\begin{thm} \label{main2} Let $G$ be isomorphic to the fundamental 
group of $M$ then  $$Wh(G)=K_{-i}(G)=\tilde K_0(G)=0$$ for 
all $i\geq 2$.\end{thm}

\section{Proofs of Theorems \ref{main1} and \ref{main2}}

For terminologies on $3$-manifolds used in the proofs see 
\cite{H} or \cite{R0}.

\begin{proof}[Proof of Theorem \ref{main2}]
By (Kneser-Milnor) prime decomposition
theorem $G$ is isomorphic to the free product of a free group and
finitely many groups $G_1,G_2,\ldots ,G_n$ where for each $i$, $G_i$ 
is isomorphic to the fundamental group of an  
aspherical irreducible 3-manifold $M_i$ (see [\cite{R1}, Lemma 3.1]). 
Since the Whitehead
group of a free product is the direct sum of the Whitehead groups of 
the individual factors of the free product (see \cite{St}) 
 it is enough to prove that the Whitehead group vanishes for $G_i$. Now by
the Geometrization Theorem (conjectured by Thurston and proved by 
Perelman) $M_i$ is either 
Seifert fibered, Haken or hyperbolic. The hyperbolic
case follows from some more general result of Farrell and Jones in \cite{FJ86}, for Haken 
case it follows from Waldhausen's
result in \cite{W}. For non-Haken Seifert fibered space the 
vanishing result is due to
Plotnick (see \cite{Pl}). For the reduced projective class groups 
$\tilde K_0(-)$ and for the negative 
$K$-groups $K_{-i}(-)$ the same sequence of arguments and references work. 
For details see \cite{FJ85}. In fact, more generally it is shown in \cite{FJ85}  
that $G$ is $K$-flat, i.e., $Wh(G\t {\Bbb Z}^n)=0$ for all non-negative 
integer $n$. 

This completes the proof of Theorem \ref{main2}.
\end{proof}

Below we recall the statement of the Fibered Isomorphism Conjecture 
of Farrell and Jones. For details about this conjecture see 
\cite{FJ}. Here we follow the formulation given
in [\cite{FL}, Appendix]. 

Let $\cal F$ be one of the three functors from the category of
topological spaces to the category of spectra: (a) the stable topological
pseudo-isotopy functor ${\cal P}()$; (b) the algebraic $K$-theory functor 
${\cal K}()$; and (c) the $L$-theory functor $L^{\l -\infty \r}()$. The 
$L$-theory functor also includes an orientation data, that is a 
homomorphism $\omega:\pi_1(X)\to {\Bbb Z}_2$. If the topological space
is an oriented manifold then this homomorphism is zero. 

Let $\cal M$ be a category whoes 
objects are continuous surjective maps $p:E\to B$ between
topological spaces $E$ and $B$. And a morphism between two maps $p:E_1\to
B_1$ and $q:E_2\to B_2$ is a pair of continuous maps $f:E_1\to E_2$,
$g:B_1\to B_2$ such that the following diagram commutes. 

$$\diagram
E_1 \rto^f \dto^p & E_2 \dto^q\\
B_1 \rto^g &B_2\enddiagram$$

There is a functor defined by Quinn in \cite{Q} from $\cal M$ to the category
of $\Omega$-spectra which associates to the map $p$ a spectrum ${\Bbb 
H}(B, {\cal F}(p))$ with the property that ${\Bbb H}(B, {\cal F}(p))={\cal
S}(E)$ if $B$ is a single point space. For an explanation of ${\Bbb H}(B,
{\cal F}(p))$ see [\cite{FJ}, Section 1.4]. Also the map ${\Bbb H}(B,
{\cal F}(p))\to {\cal F}(E)$ induced by the morphism: id$:E\to E$; 
$B\to *$ in the category $\cal M$ is called the {\it Quinn assembly map}.

Let $\Gamma$ be a discrete group and $\cal E$ be a $\Gamma$-space which
is universal for the class of all virtually cyclic subgroups of $\Gamma$ 
and denote ${\cal E}/\Gamma$ by $\cal B$. For definition and properties 
of universal 
space see [\cite{FJ}, Appendix]. Let $X$ be a space on which $\Gamma$ acts
freely and properly discontinuously and $p:X\times_{\Gamma} {\cal E}\to
{\cal E}/{\Gamma}={\cal B}$ be the map induced by the projection onto the 
second factor of $X\times {\cal E}$. 

The {\it Fibered Isomorphism Conjecture} for $\G$ states that the map $${\Bbb
H}({\cal B}, {\cal F}(p))\to {\cal F}(X\times_{\Gamma} {\cal E})={\cal
F}(X/\Gamma)$$ is an (weak) equivalence of spectra. The equality in the
above display is induced by the map $X\times_{\Gamma}{\cal E}\to X/\Gamma$
and using the fact that $\cal F$ is homotopy invariant. If $X$ is simply 
connected then this is called the {\it Isomorphism Conjecture} for $\G$.  

In this paper we consider the case when ${\cal F}()=L^{\l -\infty \r}()$. 
We have already mentioned that this $L$-theory functor contains 
the orientation data $\omega:\G \to {\Bbb Z}_2$ so as to include the 
case of nonorientable manifolds.

Let us now deduce the following theorem which is an immediate consequence 
of [\cite{R4}, $3(a)$ of Theorem 2.2] and some recent results from  
\cite{BL} and \cite{BFL}.

\begin{thm}\label{3-manifold} Let $G$ be isomorphic to the 
fundamental group of a 
$3$-manifold. Then the Farrell-Jones Fibered Isomorphism conjecture 
in $L^{\l -\infty \r}$-theory is true for $G\wr H$ where $H$ is some finite 
group.\end{thm}

\begin{proof} The theorem follows from [\cite{R4}, $3(a)$ of Theorem 2.2] 
provided we show that the conjecture is true for $\Gamma\wr H$ 
where $H$ is some finite group and $\Gamma$ belongs to the following 
classes of groups:

1). ${\Bbb Z}^2\rtimes_{\sigma} {\Bbb Z}$ for all actions $\sigma$ 
of $\Bbb Z$ on ${\Bbb Z}^2$.

2). Fundamental groups of closed nonpositively curved Riemannian 
$3$-manifolds.

3). $\G\simeq \lim_{i\in I}\G_i$ where $\{\G_i\}$ is a directed 
system of groups so that for each $i\in I$ the conjecture is 
true for $\G_i\wr K$ where $K$ is some finite group.

We now note the following to complete the proof of the 
Theorem.

(1) follows from \cite{BFL} where the conjecture is proved for 
virtually polycyclic groups.

(2) follows from \cite{BL} where the conjecture is proved for 
finite dimensional $CAT(0)$-groups.

And (3) follows from [\cite{FL}, Theorem 7.1].
\end{proof}

\begin{proof}[Proof of Theorem \ref{main1}] If $\p M\neq\emptyset$ 
then the theorem follows from [\cite{R0}, Theorem 1.1]. 
Therefore we can assume that 
$M$ is closed. Now recall that the combination of 
Theorems \ref{main2} and 
\ref{3-manifold} imply the isomorphism of the classical 
assembly map in $L$-theory. Namely, the map 
$H_k(BG, {\Bbb L}_0)\to L_k(G)$ is an isomorphism for all $k$. Since 
$M$ aspherical it is a model of $BG$, thus we have the 
isomorphism $H_k(M,{\Bbb L}_0)\to L_k(G)$. 
See the proof of [\cite{LR}, Theorem 1.28] or 
[\cite{R5}, Corollary 5.3] for a detailed argument.

Next we recall the definition of structure set and 
the surgery exact sequence.

Let $M$ be a compact manifold with boundary (may be empty) 
so that $Wh(\pi_1(M))=0$.
Consider all objects $(N, \partial N, f)$, 
where $N$ is a manifold with boundary $\partial N$ 
and $f: N\to M$ is a homotopy equivalence such that 
$f|_{\partial N}: \partial N\to \partial M$ is a 
homeomorphism. Two such objects $(N_1, \partial N_1, f_1)$ 
and $(N_2, \partial N_2, f_2)$ are equivalent if there 
is a homeomorphism $g:N_1\to N_2$ such that the obvious 
diagram commutes up to homotopy relative to the boundary. 
The equivalence classes of these objects is the homotopy-
topological structure set ${\cal S}(M, \partial M)$.

In \cite{Ra} Ranicki defined homotopy functors ${\cal S}_k(X)$ from 
the category of topological spaces to the category of abelian groups 
which fit into the following exact sequence: $$\cdots\longrightarrow 
{\cal S}_k(X)\longrightarrow H_k(X, {\Bbb L}_0)\longrightarrow 
L_k(\pi_1(X))\longrightarrow {\cal S}_{k-1}(X)\longrightarrow\cdots.$$

Also it is shown in \cite{Ra} that there is a bijection between 
${\cal S}(M\times {\Bbb D}^k, \partial (M\times {\Bbb D}^k))$ and 
${\cal S}_{k+dim M}(M)$ provided $dim M+k \geq 5$.

The proof of the theorem is now complete since we have already proved   
the isomorphism $H_k(M,{\Bbb L}_0)\to L_k(G)$ for all $k$.
\end{proof} 

\newpage
\bibliographystyle{plain}
\ifx\undefined\bysame
\newcommand{\bysame}{\leavevmode\hbox to3em{\hrulefill}\,}
\fi

\medskip

\end{document}